\documentclass[11pt,reqno]{amsart}

\usepackage{graphicx}
\usepackage{amssymb}
\usepackage{multirow}
\usepackage{slashbox}
\usepackage{amssymb}
\usepackage{graphics,multirow}
\usepackage{amsmath}
\usepackage{epsfig,multirow}
\usepackage{amsfonts}
 \usepackage{url}
\usepackage[dvips]{color}

 \usepackage{lineno}
\usepackage[figuresright]{rotating}

\newtheorem{thm}{Theorem}[section]
\newtheorem{defn}[thm]{Definition}
\newtheorem{prop}[thm]{Proposition}
\newtheorem{cor}[thm]{Corollary}
\newtheorem{lem}[thm]{Lemma}
\newtheorem{remark}[thm]{Remark}

\newcommand{\ds}{\displaystyle}

\tiny\normalsize

\begin{document}

\title{Perturbations on $K$-Fusion Frames}

\author[A. Bhandari]{A. Bhandari}
\address{Department of Mathematics\\ National Institute of Technology\\
Meghalaya, India}
\email{animesh@nitm.ac.in}

\author[S. Mukherjee]{S. Mukherjee}
\address{Department of Mathematics\\ National Institute of Technology\\
Meghalaya, India}
\email{saikat.mukherjee@nitm.ac.in}
\thanks{Second author is partially supported by DST-SERB project MTR/2017/000797}


\begin{abstract}
 $K$-fusion frames are generalizations of fusion frames in frame theory. This article characterizes various kinds of property of $K$-fusion frames. Several perturbation results on $K$-fusion frames are formulated and analyzed.
\end{abstract}

\maketitle

\section{Introduction}

The concept of Hilbert space frames was first introduced by Duffin and Schaeffer \cite{DuSc52} in 1952. Later, in 1986, frame theory was reintroduced and popularized by Daubechies, Grossman and Meyer \cite{DaGrMa86}. Since then frame theory has been widely used by mathematicians and engineers in various fields of mathematics and engineering, namely, operator theory \cite{HaLa00}, harmonic analysis \cite{Gr01}, signal processing \cite{Fe99}, sensor network \cite{CaKuLiRo07}, data analysis \cite{CaKu12}, etc.

Frame theory literature became richer through several generalizations-fusion frame (frames of subspaces) \cite{CaKu04, CaKuLi08} , $G$-frame (generalized frames) \cite{Su06}, $K$-frame (atomic systems) \cite{Ga12}, $K$-fusion frame (atomic subspaces) \cite{Bh17}, etc. and these generalizations have been proved to be useful in various applications.

This article focuses on study, characterize and explore several properties of $K$-fusion frame. It is organized as follows. Section \ref{Sec-Preli} is devoted to the basic definitions and results related to frames and fusion frames. The characteristics of $K$-fusion frames are discussed in Section \ref{Sec-char}. Finally, results related to perturbation and erasure properties are established  in Section \ref{Sec-per}.

Throughout the paper, $\mathcal{H}$ is a separable Hilbert space. We denote by $\mathcal{L}(\mathcal{H}_1, \mathcal{H}_2)$ the space of all bounded linear operators from  $\mathcal{H}_1$ into $\mathcal{H}_2$, and $\mathcal L(\mathcal H)$ for $\mathcal L(\mathcal H, \mathcal H)$. For $T \in \mathcal{L}\mathcal{(H)}$, we denote $D(T), N(T)$ and $R(T)$ for domain, null space and range of $T$, respectively. For a collection of closed subspaces $\mathcal W_i$ of $\mathcal H$ and scalars $v_i$, $i\in I$, the weighted collection of closed subspaces $\lbrace (\mathcal W_i , v_i) \rbrace_{i \in I}$ is denoted by $\mathcal W_v$. We consider the index set $I$ to be finite or countable.

\section{Preliminaries}\label{Sec-Preli}

In this section we recall basic definitions and results needed in this paper. We refer the books of Ole Christensen \cite{Ch08} and Casazza et.al. \cite{CaKu12} for an introduction to frame theory.

\subsection{Frame} A collection  $\{ f_i \}_{i\in I}$ in $\mathcal{H}$ is called a \emph{frame} if there exist constants $A,B >0$ such that \begin{equation}\label{Eq:Frame} A\|f\|^2~ \leq ~\sum_{i\in I} |\langle f,f_i\rangle|^2 ~\leq ~B\|f\|^2,\end{equation}for all $f \in \mathcal{H}$. The numbers $A, B$ are called \emph{frame bounds}. The supremum over all $A$'s and infimum over all $B$'s satisfying above inequality are called the \emph{optimal frame bounds}.
If a collection satisfies only the right inequality in (\ref{Eq:Frame}), it is called a {\it Bessel sequence}.

Given a frame $\{f_i\}_{i\in I}$ of $\mathcal{H}$. The \emph{pre-frame operator} or \emph{synthesis operator} is a bounded linear operator $T: l^2(I)\rightarrow \mathcal{H}$ and is defined by $T\{c_i\} = \ds \sum_{i\in I} c_i f_i$. The adjoint of $T$, $T^*: \mathcal{H} \rightarrow l^2(I)$, given by $T^*f = \{\langle f, f_i\rangle\}$, is called the \emph{analysis operator}. The \emph{frame operator}, $S=TT^*: \mathcal{H}\rightarrow \mathcal{H}$, is defined by $$Sf=TT^*f = \sum_{i\in I} \langle f, f_i\rangle f_i.$$ It is well-known that the frame operator is bounded, positive, self adjoint and invertible.



\subsection{Fusion Frame} Consider a weighted collection of closed subspaces, $\mathcal W_v$, of $\mathcal H$. Then $\mathcal W_v$  is said to be a fusion frame for $\mathcal H$, if there exist constants $0<A \leq B <\infty$ satisfying
\begin{equation}\label{Eq:Fus-frame}
A\|f\|^2 \leq \sum_{i \in I} v_i ^2 \|P_{\mathcal W_i}f\|^2 \leq B\|f\|^2, \end{equation}
where $P_{\mathcal W_i}$ is the orthogonal projection from $\mathcal H$ onto $\mathcal W_i$.
The constants $A$ and $B$ are called {\it fusion frame bounds}.
A collection of closed subspaces, satisfying only the right inequality in (\ref{Eq:Fus-frame}), is called a {\it fusion Bessel sequence}.

Additionally, it is to be noted that for every fusion frame $\mathcal W_v$, there are frames $\lbrace f_{ij} \rbrace_{j \in J_i}$ for each $\mathcal W_i$ and these are called local frames for $\mathcal W_v$. It is a well-known fact that $\mathcal W_v$ is a fusion frame if and only if $\lbrace w_i f_{ij} \rbrace_{j \in J_i, i \in I}$ is a frame for $\mathcal H$, for details readers are referred to Theorem 3.2 in \cite{CaKu04}.

For a family of closed subspaces, $\lbrace \mathcal W_i \rbrace_{i \in I}$, of $\mathcal H$, the associated $l^2$ space is defined by $\left(\sum\limits_{i \in I} \bigoplus \mathcal W_i\right)_{l^2} = \lbrace \lbrace f_i \rbrace_{i \in I} : f_i \in \mathcal W_i, \sum\limits_{i \in I} \|f_i\|^2 < \infty \rbrace$ with the inner product $\langle \lbrace f_i \rbrace, \lbrace g_i \rbrace \rangle = \sum\limits_{i \in I} \langle f_i , g_i \rangle_{\mathcal{H}} $.

Let $\mathcal W_v$ be a fusion frame. Then the associated synthesis operator $T_{\mathcal W}: (\sum \limits_{i \in I} \bigoplus \mathcal W_i)_{l^2} \rightarrow \mathcal H$ is defined as $T_{\mathcal W} (f)=\sum \limits_{i \in I} v_i f_i$ for all $\lbrace f_i \rbrace_{i \in I} \in (\sum \limits_{i \in I} \oplus \mathcal W_i)_{l^2}$ and the analysis operator $T^*_{\mathcal W}: \mathcal H \rightarrow  (\sum \limits_{i \in I} \oplus \mathcal W_i)_{l^2} $ is defined as  $T^*_{\mathcal W} (f)=\lbrace v_i P_{\mathcal W_i} (f) \rbrace_{i \in I} $. It is well-known that (see \cite{CaKu04}) the synthesis operator $ T_{\mathcal W}$ of a fusion frame is bounded, linear and onto, whereas the corresponding analysis operator $ T_{\mathcal W} ^*$  is (possibly into) an isomorphism. Corresponding fusion frame operator is defined as $S_\mathcal W (f) = T_{\mathcal W}T_{\mathcal W}^*(f)=\sum \limits_{i \in I} v_i ^2 P_{\mathcal W_i}(f)$.  $S_\mathcal{W}$ is bounded, positive, self adjoint and invertible. Any signal $f \in \mathcal H$ can be expressed by its fusion frame measurements $\lbrace v_i P_{\mathcal W_i} f \rbrace_{i \in I}$ as
\begin{equation}\label{Eq:Fusion-frame-recons} f=\sum_{i \in I} v_i S_\mathcal W ^{-1} (v_i P_{\mathcal W_i} f) .\end{equation}

\subsection{$K$-fusion frame}
In \cite{Bh17}, authors, introduced a generalization of fusion frame, $K$-fusion frame, and scrutinized the equivalence between atomic subspaces and $K$-fusion frames. $K$-fusion frame is used to reconstruct signals from range of a bounded linear operator $K$.

\begin{defn}($K$-fusion frame)
Let $K\in \mathcal L(\mathcal H)$,  $\mathcal W_v=\lbrace \mathcal (W_i,v_i) \rbrace_{i \in I}$ be a weighted collection of closed subspaces of $\mathcal H$. Then   $\mathcal W_v$ is said to be a $K$-fusion frame for $\mathcal H$  if there exist positive constants $A, B$ such that for all $f \in \mathcal H$ we have
\begin{equation}\label{Eq:K-Fus-frame}
A\|K^*f\|^2 \leq \sum_{i \in I} v_i ^2 \|P_{\mathcal W_i}f\|^2 \leq B\|f\|^2. \end{equation}
\end{defn}

In the rest of this Section, we recall some fundamental results in Hilbert space that are necessary to present some outcomes of this article.

\begin{thm}(Douglas' factorization theorem \cite{Do66})\label{Thm-Douglas} Let $\mathcal{H}_1, \mathcal{H}_2,$ and $\mathcal H$ be Hilbert spaces and $S\in \mathcal L(\mathcal H_1, \mathcal H)$, $T\in \mathcal L(\mathcal H_2, \mathcal H)$. Then the following are equivalent:
\begin{enumerate}
\item $R(S) \subseteq R(T)$.
\item $SS^* \leq \alpha TT^*$ for some $\alpha >0$.
\item $S = TL$ for some $L \in \mathcal L(\mathcal H_1, \mathcal H_2)$.
\end{enumerate}
\end{thm}

\begin{lem}(Moore-Penrose pseudo-inverse \cite{Ka80, Ru91, Ch08, As09}) Let $\mathcal H$ and $\mathcal K$ be two Hilbert spaces and $T \in \mathcal L(\mathcal H,\mathcal K)$ be a closed range operator, then the followings hold:
\begin{enumerate}
\item $TT^{\dag}=P_{T(\mathcal H)}$, $T^{\dag}T=P_{T^*(\mathcal K)}$
\item $\frac {\|f\|} {\|T^{\dag}\|} \leq \|T^*f\|$ for all $f \in T(\mathcal H)$.
\item $TT^{\dag}T=T$, $T^{\dag}TT^{\dag}=T^{\dag}$, $(TT^{\dag})^*=TT^{\dag}$, $(T^{\dag}T)^*=T^{\dag}T$.
\end{enumerate}
\end{lem}


\begin{lem}(\cite{Ga07, LiYaZh15}) \label{lem-projection} Suppose  $\mathcal H$ and $\mathcal K$ be two Hilbert spaces and $T \in \mathcal L(\mathcal H,\mathcal K)$. consider $\mathcal W$ be a closed subspace of $\mathcal H$ and $\mathcal V$ be a closed subspace of $\mathcal K$. Then we have the followings:
\begin{enumerate}
\item $P_{\mathcal W}T^*P_{\overline {T \mathcal W}}=P_{\mathcal W}T^*$.
\item $P_{\mathcal W}T^*P_{\mathcal V}=P_{\mathcal W}T^*$. if and only if $T \mathcal W \subset \mathcal V$.
\end{enumerate}
\end{lem}


\begin{defn}(Drazin inverse \cite{Ad03, HuZhGeYu12}) \label{def-Drazin} Let $S, T \in \mathcal L(\mathcal H)$, $S$ is said to be the Drazin inverse of $T$ if we have the following:
\begin{enumerate}
\item $STS=S$.
\item $ST=TS$.
\item $TST^k=T^k$, for some positive integer $k$.
\end{enumerate}
\end{defn}

It is to be noted that $T \in \mathcal L(\mathcal H)$ has the Drazin inverse in $ \mathcal L(\mathcal H)$ if and only if $\lambda =0$ is a pole of the resolvent operator $(\lambda I - T)^{-1}$. Moreover, the order of the pole is equal to the index of $T$. In particular $0$ is not an accumulation point in the spectrum $\sigma(T)$.

\section{Characterization of $K$-fusion frames}\label{Sec-char}
In this section, we characterize various properties of $K$-fusion frame. The following theorem provides a sufficient condition on a bounded, linear operator $K$ under which the image of $K$-fusion frame remains a $K$-fusion frame.

\begin{thm} Let $ K \in \mathcal L(\mathcal H)$ be an idempotent, closed range operator and  $\mathcal W_v$  be a $K$-fusion frame for $\mathcal H$ with $K^{\dag} \overline {K(\mathcal W_i)} \subset \mathcal W_i$. Then $\lbrace (K \mathcal W_i, v_i)\rbrace_{i \in I}$ constitutes a $K$-fusion frame for $\mathcal H$.
\end{thm}

\begin{proof} First we prove for all $i \in I$, $K(\mathcal W_i)$ is a closed subspace in  $\mathcal H$. Since $K^{\dag} \overline {K(\mathcal W_i)} \subset \mathcal W_i$, $K K^{\dag} \overline {K(\mathcal W_i)} \subset K(\mathcal W_i)$. Therefore using the Lemma 2.5.2 in \cite{Ch08} we have $KK^* (KK^*)^{-1}\overline {K(\mathcal W_i)} \subset K(\mathcal W_i)$ and hence $K(\mathcal W_i)$ is a closed subspace in  $\mathcal H$ for all $i \in I$. Since $\lbrace (\mathcal W_i, v_i)\rbrace_{i \in I}$ is a $K$-fusion frame in $\mathcal H$, there exist $A, B>0$ such that for all $f \in \mathcal H$ we have, $$A \|K^*f\|^2 \leq \sum_{i \in I} v_i ^2 \|P_{\mathcal W_i}f\|^2 \leq B \|f\|^2.$$ Again as $K$ is idempotent, using the Lemma \ref{lem-projection} we obtain, $$A \|K^*f\|^2 \leq \sum_{i \in I} v_i ^2 \|P_{\mathcal W_i} K^*f\|^2 \leq \|K\|^2 \sum_{i \in I} v_i ^2 \|P_{K \mathcal W_i} f\|^2 .$$ Therefore $\frac {A} {\|K\|^2} \|K^*f\|^2 \leq \sum \limits_{i \in I} v_i ^2 \|P_{K \mathcal W_i} f\|^2$.

Again from the Lemma \ref{lem-projection} we have $P_{K \mathcal W_i} = P_{K \mathcal W_i} K^{\dag *} P_{\mathcal W_i} K^*$.  Therefore for all $f \in \mathcal H$ we obtain, $$\sum_{i \in I} v_i ^2 \|P_{K \mathcal W_i} f\|^2 \leq \|K^{\dag *}\|^2 \sum_{i \in I} v_i ^2 \|P_{\mathcal W_i} K^*f\|^2 \leq  B \|K^{\dag *}\|^2 \|K\|^2 \|f\|^2.$$ Hence our assertion is tenable.

\end{proof}

In the next result, we further characterize $K$-fusion frame by means of Drazin inverse.

\begin{lem} Let $ K\in \mathcal L(\mathcal H)$  has non-zero Drazin inverse, $S$. Also suppose that  $\mathcal W_v$ is a $K$-fusion frame for $\mathcal H$. Then the following hold:

\begin{enumerate}
\item  $\mathcal W_v$ is a $SKS$-fusion frame for $\mathcal H$.
\item $\mathcal W_v$ is a $SK$-fusion frame in $\mathcal H$.
\end{enumerate}
\end{lem}

\begin{proof} Since $\mathcal W_v$ is a $K$-fusion frame for $\mathcal H$, there exist $A, B>0$ such that for all $f \in \mathcal H$ we have, $A \|K^*f\|^2 \leq \sum \limits_{i \in I} v_i ^2 \|P_{\mathcal W_i}f\|^2 \leq B \|f\|^2$. Again as $S$ is the non-zero Drazin inverse of $K$, for all $f \in \mathcal H$ we have

$$\frac {A} {\|S\|^ {4}} \|(SKS)^* f\|^2 \leq A \|K^*f\|^2 \leq \sum_{i \in I} v_i ^2 \|P_{\mathcal W_i}f\|^2 \leq B \|f\|^2.$$ and also $$\frac {A} {\|S\|^2} \|(SK)^* f\|^2 \leq A \|K^*f\|^2 \leq \sum_{i \in I} v_i ^2 \|P_{\mathcal W_i}f\|^2 \leq B \|f\|^2.$$
Hence the conclusions follow.
\end{proof}

\begin{remark} It is to be noted that $\mathcal W_v$ is also a $KS$-fusion frame for $\mathcal H$ but this result is obvious as for any $(0\neq) T \in \mathcal L(\mathcal H)$,  $$\frac {A} {\|T\|^2} \|(KT)^*f\|^2 \leq A \|K^*f\|^2 \leq \sum_{i \in I} v_i ^2 \|P_{\mathcal W_i}f\|^2 \leq B \|f\|^2,$$ for all $f \in \mathcal H$.

\end{remark}

In the following theorem we scrutinize the robustness of $K$-fusion frames under erasure property.

\begin{thm}
 Let $ K \in \mathcal L(\mathcal H)$ be a closed range operator and  $\mathcal W_v$ be a $K$-fusion frame for $\mathcal H$ with bounds $A$ and $B$. Suppose $J\subseteq I$ such that $\sum\limits_{i \in J} v_i ^2 = C<\infty$ with $(A- C \|K^{\dag}\|^2) >0$. Then $\lbrace (\mathcal W_i, v_i)\rbrace_{i \in {I \setminus J}}$ forms a $K$-fusion frame for $R(K)$ with bounds $(A- C \|K^{\dag}\|^2)$ and $B$.
\end{thm}

\begin{proof} Since $K$ has closed range, for all $f \in  R(K)$ we have for all $i\in I$,
$$\|P_{\mathcal W_i} f \| \leq \|K^{\dag}\| \|K^*f\|$$ and hence
 $\sum\limits_{i\in J}v_i^2\|P_{\mathcal W_i}f\|^2 \leq C \|K^{\dag}\|^2 \|K^*f\|^2~\text{for all}~f \in R(K).$ Consequently for all $f \in R(K)$,
 $$(A- C \|K^{ \dag}\|^2) \|K^* f\|^2 = A\|K^* f\|^2 - C \|K^{\dag}\|^2 \|K^*f\|^2  \leq \sum_{i\in {I \setminus J}}v_i^2\|P_{\mathcal W_i}f\|^2.$$
The upper bound follows directly from the assumption.
\end{proof}

\section{Perturbation Properties}\label{Sec-per}

In this section we analyze stability conditions of $K$-fusion frames under perturbations.

\begin{lem} \label{Sec4-lem1} Let $ K_1 \in \mathcal L(\mathcal H)$, $\mathcal W_v$ be a $K_1$-fusion frame for $\mathcal H$. Suppose  $ K_2 \in \mathcal L(\mathcal H)$ and $a, b\geq 0$ such that $$\|(K_1 ^* - K_2 ^*)f\| \leq a\|K_1 ^* f\| + b \|K_2 ^* f\|,~~\text{for all}~f \in \mathcal H.$$ Then $\mathcal W_v$ is also a $K_2$-fusion frame for $\mathcal H$ if $b<1$.

\end{lem}

\begin{proof} Since $\mathcal W_v$ is $K_1$-fusion frame for $\mathcal H$, there exist $A~, B>0$, for all $f \in \mathcal H$ we have $A \|K_1 ^*f\|^2 \leq \sum\limits_{i \in I} v_i ^2 \|P_{\mathcal W_i}f\|^2 \leq B \|f\|^2$. Now for all $f \in \mathcal H$ we obtain
\begin{eqnarray*}
\|K_2 ^*f\| \leq \|(K_1 ^* - K_2 ^*)f\| + \|K_1 ^*f\| \leq (1+a)\|K_1 ^*f\| + b\|K_2 ^*f\|.
\end{eqnarray*}
Therefore for all $f \in \mathcal H$ we have
\begin{eqnarray*}
A\left(\frac {1-b} {1+a}\right)^2    \|K_2 ^*f\|^2 \leq A \|K_1 ^*f\|^2 \leq \sum\limits_{i \in I} v_i ^2 \|P_{\mathcal W_i}f\|^2 \leq B \|f\|^2.
\end{eqnarray*}
Hence our assertion is tenable.

\end{proof}

\begin{cor}  Let $ K_1~,K_2 \in \mathcal L(\mathcal H)$ and $0\leq a,~ b < 1$ so that for all $f \in \mathcal H$, $$\|(K_1 ^* - K_2 ^*)f\| \leq a\|K_1 ^* f\| + b \|K_2 ^* f\|.$$ Also suppose, $\mathcal W_v$ is a weighted collection of closed subspaces of $\mathcal H$. Then $\mathcal W_v$ is $K_1$-fusion frame for $\mathcal H$ if and only if it is $K_2$-fusion frame for $\mathcal H$.

\end{cor}

\begin{proof} The proof follows from above Lemma \ref{Sec4-lem1} by interchanging the roles of $K_1$ and $K_2$.


\end{proof}

Above results immediately provide the following result.
\begin{cor}  Let $ K \in \mathcal L(\mathcal H)$, $\mathcal W_v$ be  $K$-fusion frame for $\mathcal H$ with $$\|K^* f-f\| \leq a \|K^* f\| + b \|f\|,~~ \text{for all}~f \in \mathcal H,$$ where $0\leq a,~ b <1$. Then $\mathcal W_v$ is fusion frame for $\mathcal H$.

\end{cor}

Given a fusion Bessel sequence or $K$-fusion frame or fusion frame, can we construct a $K$-fusion frame? Following results address this using perturbations on projection operators.

\begin{thm}\label{Sec4-thm1} Let $a, b, c\geq 0$ and $\mathcal W_w, \mathcal V_v$ be two weighted collections of closed subspaces of $\mathcal H$ so that for all $f\in \mathcal H$, $$\left(\sum_{i \in I} \|(w_i P_{\mathcal W_i} - v_i P_{\mathcal V_i})f \|^2 \right)^{\frac {1} {2}}\leq  a \left(\sum_{i \in I} w_i ^2 \|P_{\mathcal W_i}f\|^2\right)^{\frac {1} {2}} + b  \left(\sum_{i \in I} v_i ^2 \|P_{\mathcal V_i}f\|^2\right)^{\frac {1} {2}} + c ~\Lambda(f),$$ for some $\Lambda:\mathcal H\rightarrow \mathbb R_+$. Then the following results hold:
\begin{enumerate}
\item Let $\mathcal W_w$ be a fusion Bessel sequence in $\mathcal H$ and $b<1$, $c=0$. Then $\mathcal V_v$ is a $K$-fusion frame for $\mathcal H$ for any $ K \in \mathcal L(\mathcal H)$ satisfying $R(K) \subseteq R(T_{\mathcal V})$, where $T_{\mathcal V}$ is the associated synthesis operator of   $\mathcal V_v$.

\item\label{Sec4-thm1-2}  Let $ K \in \mathcal L(\mathcal H)$, $\mathcal W_w$ be a $K$-fusion frame for $\mathcal H$ with bounds $A,~ B>0$ and $\Lambda(f)=\|K^*f\|$. Then if $a <1,~ b < 1,~ 0\leq \frac{c}{1-a}<\sqrt{A}$, $\mathcal V_v$ is also a  $K$-fusion frame for $\mathcal H$ with bounds $\left(\frac {\sqrt A (1-a) -c} {(1+b)}\right)^2$ and $\left(\frac {(1+a) \sqrt B + c \|K\|} {(1-b)}\right)^2$.

\item Let $\mathcal W_w$ be a fusion frame for $\mathcal H$ with bounds $A,~  B>0$ and $\Lambda(f) = \|f\|$. Then $\mathcal V_v$ forms a fusion frame for $\mathcal H$ with bounds $\left(\frac {\sqrt A -c-a\sqrt B} {1+b}\right)^2$ and $\left(\frac {(1+a)\sqrt B +c} {1-b}\right)^2$, if $a\sqrt B + c <\sqrt A, ~ b < 1$.
\end{enumerate}
\end{thm}

\begin{proof}
\begin{enumerate}
\item  Since $\mathcal W_w$ is a fusion Bessel sequence in $\mathcal H$, there exists $B>0$ such that for all $f \in \mathcal H$ we have $\sum\limits_{i \in I} w_i ^2 \|P_{\mathcal W_i}f\|^2 \leq B \|f\|^2$. Using Minkowski's inequality, for all $f \in \mathcal H$, we obtain

\begin{eqnarray*}
\left(\sum_{i \in I} v_i ^2 \|P_{\mathcal V_i}f\|^2\right)^{\frac {1} {2}} &\leq & \left(\sum_{i \in I}\|(w_i P_{\mathcal W_i} - v_i P_{\mathcal V_i})f \|^2 \right)^{\frac {1} {2}} +  \left(\sum_{i \in I} w_i ^2 \|P_{\mathcal W_i}f\|^2\right)^{\frac {1} {2}} \\ &\leq & (1+a)  \left(\sum_{i \in I} w_i ^2 \|P_{\mathcal W_i}f\|^2\right)^{\frac {1} {2}} + b  \left(\sum_{i \in I} v_i ^2 \|P_{\mathcal V_i}f\|^2\right)^{\frac {1} {2}},
\end{eqnarray*}
Hence $$\sum_{i \in I} v_i ^2 \|P_{\mathcal V_i}f\|^2 \leq  \sqrt B \left(\frac {1+a} {1-b}\right)^2 \|f\|^2,~~~\forall~ f\in \mathcal H.$$
Therefore $T_{\mathcal V}^*$ and hence $T_{\mathcal V}$ is well-defined.
The left hand inequality directly follows from Theorem \ref{Thm-Douglas}~.

\item  Since $\mathcal W_w$ is a $K$-fusion frame for $\mathcal H$ with bounds $A , B >0$, for all $f \in \mathcal H$ we have $A \|K^*f\|^2 \leq \sum\limits_{i \in I} w_i ^2 \|P_{\mathcal W_i}f\|^2 \leq B \|f\|^2$. Now
\begin{eqnarray*}
\left(\sum_{i \in I} v_i ^2 \|P_{\mathcal V_i}f\|^2\right)^{\frac {1} {2}} &\leq & \left(\sum_{i \in I}\|(w_i P_{\mathcal W_i} - v_i P_{\mathcal V_i})f \|^2 \right)^{\frac {1} {2}} +  \left(\sum_{i \in I} w_i ^2 \|P_{\mathcal W_i}f\|^2\right)^{\frac {1} {2}} \\ &\leq & (1+a)  \left(\sum_{i \in I} w_i ^2 \|P_{\mathcal W_i}f\|^2\right)^{\frac {1} {2}} + b  \left(\sum_{i \in I} v_i ^2 \|P_{\mathcal V_i}f\|^2\right)^{\frac {1} {2}} + c \|K^*f\|,
\end{eqnarray*}
for all $f \in \mathcal H$. Hence
$$\sum_{i \in I} v_i ^2 \|P_{\mathcal V_i}f\|^2 \leq \left(\frac {(1+a) \sqrt B + c \|K\|} {(1-b)}\right)^2 \|f\|^2,~~~\forall~ f\in \mathcal H.$$

Similarly, for the lower bound we have,
$$\left(\sum_{i \in I} w_i ^2 \|P_{\mathcal W_i}f\|^2\right)^{\frac {1} {2}} \leq a \left(\sum_{i \in I} w_i ^2 \|P_{\mathcal W_i}f\|^2\right)^{\frac {1} {2}} + (1+b) \left(\sum_{i \in I} v_i ^2 \|P_{\mathcal V_i}f\|^2\right)^{\frac {1} {2}}  + c \|K^*f\|,$$ for all $f \in \mathcal H$. Hence we obtain
$$\left(\frac {\sqrt A (1-a) -c} {(1+b)}\right)^2 \|K^*f\|^2 \leq \sum_{i \in I} v_i ^2 \|P_{\mathcal V_i}f\|^2, ~~\forall~f\in \mathcal H.$$

\item  Since $\mathcal W_w$ is a fusion frame for $\mathcal H$ with bounds $A , B >0$, for all $f \in \mathcal H$ we have $A \|f\|^2 \leq \sum\limits_{i \in I} w_i ^2 \|P_{\mathcal W_i}f\|^2 \leq B \|f\|^2$. Also for all $f \in \mathcal H$ we obtain
\begin{eqnarray*}
\left(\sum_{i \in I} v_i ^2 \|P_{\mathcal V_i}f\|^2\right)^{\frac {1} {2}} &\geq &  \left(\sum_{i \in I} w_i ^2 \|P_{\mathcal W_i}f\|^2\right)^{\frac {1} {2}} - \left(\sum_{i \in I}\|(w_i P_{\mathcal W_i} - v_i P_{\mathcal V_i})f \|^2 \right)^{\frac {1} {2}} \\ &\geq & (\sqrt A - c) \|f\| - a  \left(\sum_{i \in I} w_i ^2 \|P_{\mathcal W_i}f\|^2\right)^{\frac {1} {2}} - b  \left(\sum_{i \in I} v_i ^2 \|P_{\mathcal V_i}f\|^2\right)^{\frac {1} {2}},
\end{eqnarray*}
Therefore  for all $f \in \mathcal H$ we obtain $$\left (\frac {\sqrt A -c-a\sqrt B} {1+b}\right)^2 \|f\|^2 \leq \sum_{i \in I} v_i ^2 \|P_{\mathcal V_i}f\|^2.$$

Moreover, for all $f \in \mathcal H$ we have
\begin{eqnarray*}
\left(\sum_{i \in I} v_i ^2 \|P_{\mathcal V_i}f\|^2\right)^{\frac {1} {2}} &\leq & \left(\sum_{i \in I}\|(w_i P_{\mathcal W_i} - v_i P_{\mathcal V_i})f \|^2 \right)^{\frac {1} {2}} +  \left(\sum_{i \in I} w_i ^2 \|P_{\mathcal W_i}f\|^2\right)^{\frac {1} {2}} \\ &\leq & (1+a)  \left(\sum_{i \in I} w_i ^2 \|P_{\mathcal W_i}f\|^2\right)^{\frac {1} {2}} + b  \left(\sum_{i \in I} v_i ^2 \|P_{\mathcal V_i}f\|^2\right)^{\frac {1} {2}} + c \|f\|,
\end{eqnarray*}
Hence for all $f \in \mathcal H$ we obtain $$\sum_{i \in I} v_i ^2 \|P_{\mathcal V_i}f\|^2 \leq \left (\frac {(1+a)\sqrt B +c} {1-b}\right)^2 \|f\|^2.$$

\end{enumerate}
\end{proof}
We acknowledge that recently Li and Leng \cite{LiLe18} proved a similar result as stated in the second statement of above theorem. We present the result here as this work has been done almost simultaneously with the work of Li and Leng.

In the following proposition we discuss another perturbation condition on the projection operators to obtain a $K$-fusion frame.

\begin{prop}  Let $ K \in \mathcal L(\mathcal H)$, $\mathcal W_w$ be a $K$-fusion frame for $\mathcal H$ with bounds $A,~ B>0$. Also suppose $\mathcal V_v$  is any weighted collection of closed subspaces  of $\mathcal H$ so that for all $f \in \mathcal H$,
$$\sum_{i \in I} |\langle f, \left(w_i ^2 P_{\mathcal W_i} - v_i^2 P_{\mathcal V_i}\right)f \rangle| \leq R \|K^* f\|^2,~~where~~ 0<R<A.$$ Then $\mathcal V_v$  forms a $K$-fusion frame for $\mathcal H$ with bounds $(A-R)~and~ (B+R\|K\|)$.

\end{prop}

\begin{proof} Since $\mathcal W_w$ is a $K$-fusion frame for $\mathcal H$ with bounds $A, B>0$, we have $$A \|K^*f\|^2 \leq \sum\limits_{i \in I} v_i ^2 \|P_{\mathcal W_i}f\|^2 \leq B \|f\|^2,~~~\forall ~f \in \mathcal H.$$ Now for all $f \in \mathcal H$ we obtain
\begin{eqnarray*}
\sum_{i \in I} w_i^2\|P_{\mathcal W_i}f\|^2 &\leq& \sum_{i \in I} |\langle f, \left(w_i ^2 P_{\mathcal W_i} - v_i^2 P_{\mathcal V_i}\right)f \rangle| + \sum_{i \in I} v_i^2\|P_{\mathcal V_i}f\|^2  \\
&\leq& R \|K^* f\|^2 + \sum_{i \in I} v_i^2\|P_{\mathcal V_i}f\|^2
\end{eqnarray*}
Therefore for all $f \in \mathcal H$ we have, $(A-R)\|K^* f\|^2 \leq \sum\limits_{i \in I} v_i ^2 \|P_{\mathcal V_i}f\|^2$.

Similarly, $\sum\limits_{i \in I} v_i ^2 \|P_{\mathcal V_i}f\|^2 \leq (B+R\|K\|) \|f\|^2$, for  all $f \in \mathcal H$.

\end{proof}

In the following two results we analyze perturbation conditions under which a fusion Bessel sequence forms a $K$-fusion frame.

\begin{thm}  Let $\mathcal W_w$ be a fusion Bessel sequence in $\mathcal H$ with bound $B>0$ and $J \subsetneq I$ with $T_{\mathcal W}$ is the associated synthesis operator of $\lbrace (\mathcal W_i, w_i)\rbrace_{i \in {I \setminus J}}$. Let $a,~b\geq 0$ and $ K \in \mathcal L(\mathcal H)$ satisfying $\|(K^* - T_\mathcal W T_{\mathcal W}^*)f\| \leq a~ \|K^*f\| + b~ \|T_{\mathcal W}^* f\|$ for all $f \in \mathcal H$. Then $\lbrace (\mathcal W_i, w_i)\rbrace_{i \in {I \setminus J}}$ forms a $K$-fusion frame for $\mathcal H$ with bounds $\left(\frac {1-a} {b+\|T_\mathcal W\|}\right)^2$ and $B$ if $a< 1$.
\end{thm}

\begin{proof} We have for  all $f \in \mathcal H$, $\|K^*f\| \leq \|(K^* - T_\mathcal WT_{\mathcal W}^*)f\| + \|T_\mathcal WT_{\mathcal W}^*f\| \leq  a \|K^*f\| + (b+\|T_\mathcal W\|)\|T_{\mathcal W}^* f\|$. Therefore $$\left(\frac {1-a} {b+\|T_\mathcal W\|}\right)^2 \|K^*f\|^2 \leq  \sum_{i \in {I\setminus J}} w_i^2 \|P_{\mathcal W_i}f\|^2  \leq \sum_{i \in I} w_i^2 \|P_{\mathcal W_i}f\|^2 \leq B \|f\|^2,$$ for all $f\in \mathcal H$.

\end{proof}

\begin{thm} Let $\mathcal W_w$ be a fusion Bessel sequence in $\mathcal H$ with bound $B>0$ and let  $J \subsetneq I$ so that the associated  synthesis operator of $\lbrace (\mathcal W_i, w_i)\rbrace_{i \in {I \setminus J}}$  is $T_{\mathcal W}$. Let $a,~b,~c\geq 0$ and $K \in \mathcal L(\mathcal H)$ be a closed range operator such that $\|(K^* - T_{\mathcal W}T_{\mathcal W}^*)f\| \leq a~ \|K^*f\| + b~ \|T_{\mathcal W}^* f\|+c~\|f\|$ for all $f \in \mathcal H$. Then if $a+c\|K^{\dag}\| <1$, $\lbrace (\mathcal W_i, w_i)\rbrace_{i \in {I \setminus J}}$ is a $K$-fusion frame for $R(K)$ with bounds $\left(\frac {1-a-c\|K^{\dag}\|} {b+\|T_\mathcal W\|}\right)$ and $B$.
\end{thm}

\begin{proof} For all $f \in \mathcal H$ we have, $\|K^*f\| \leq \|(K^* - T_{\mathcal W}T_{\mathcal W}^*)f\| + \|T_{\mathcal W}T_{\mathcal W}^*f\| \leq  a~ \|K^*f\| + \left(b+\|T_{\mathcal W}\|\right)\|T_{\mathcal W}^* f\|+c~\|f\|$ and hence for all $f \in R(K)$,
$$\left(1-a-c\|K^{\dag}\|\right)\|K^*f\| \leq \left(b+\|T_{\mathcal W}\|\right)\|T_{\mathcal W}^* f\|.$$ Therefore for all $f \in R(K)$, we have the following: $$\left(\frac{1-a-c\|K^{\dag}\|} {b+\|T_\mathcal W\|}\right)~ \|K^*f\| \leq \|T_{\mathcal W}^* f\| = \sum_{i \in {I\setminus J}} w_i^2 \|P_{\mathcal W_i}f\|^2 \leq \sum_{i \in I} w_i^2 \|P_{\mathcal W_i}f\|^2 \leq B \|f\|^2.$$ Consequently, our declaration is sustainable.
\end{proof}

\section*{Acknowledgment}

The first author acknowledges the financial support of MHRD, Government of India.

\bibliographystyle{amsplain}
\bibliography{references-frame}

\end{document}